\documentclass[a4paper,11pt]{amsart}
\usepackage{amsmath,amssymb,amsfonts}
\usepackage{epsfig}
\usepackage{mathrsfs}
\usepackage{setspace}
\newtheorem{defn}{Definition}[section]
\newtheorem{thm}{Theorem}[section]

\newtheorem{cor}{Corollary}[section]
\newtheorem{rmk}{Remark}[section]
\newtheorem{lma}{Lemma}[section]

\newtheorem{exm}{Example}[section]
\def\N{{\rm I\kern-0.16em N}}
\def\R{{\rm I\kern-0.16em R}}
\def\E{{\rm I\kern-0.16em E}}
\def\P{{\rm I\kern-0.16em P}}
\def\F{{\rm I\kern-0.16em F}}
\def\B{{\rm I\kern-0.16em B}}
\def\C{{\rm I\kern-0.46em C}}
\def\G{{\rm I\kern-0.50em G}}
\newcommand{\ud}{\mathrm{d}}
\numberwithin{equation}{section}

\usepackage{ae}

\begin{document}

\let\oldthefootnote\thefootnote
\footnotetext[1]{Department of Mathematics and System Analysis, Aalto University School of Science, Helsinki,
P.O. Box 11100, FIN-00076 Aalto, FINLAND, lauri.viitasaari\@ aalto.fi}
\let\thefootnote\oldthefootnote

\let\oldthefootnote\thefootnote
\renewcommand{\thefootnote}{\fnsymbol{footnote}}
\footnotetext[1]{Acknowledgements: The author thanks Esko Valkeila for discussions and comments which improved the paper. 
The author thanks the Finnish Doctoral Programme in Stochastics and Statistics for financial support.}
\let\thefootnote\oldthefootnote

\title[Rate of convergence]
{Rate of convergence for discrete approximation of option prices}

\author[Viitasaari]{Lauri Viitasaari$^{1,*}$}

\doublespacing
\begin{abstract}
In this article, we study the rate of convergence of prices when a model is approximated by some simplified model.
We also provide a method how explicit error formula for more general options can be obtained if such formula is available for 
digital option prices. We illustrate our results by considering convergence of 
binomial prices to Black-Scholes prices. We also consider smooth convergence in which the approximation does not oscillate 
for general class of payoff functions.

\medskip


\smallskip

\end{abstract}

\maketitle

\section{Introduction}
Fundamental theorem of asset pricing states that arbitrage-free prices of claims can be computed by expectation with respect
to an equivalent martingale measure. However, usually the prices cannot be computed in closed form and one has to use numerical
methods. These numerical methods can roughly be divided into three different categories, which all have efficiencies and drawbacks; tree-methods, Monte Carlo-methods and 
methods based on partial different equations (PDE) or partial integro-differential equations (PIDE). Especially, an upside on tree-methods
is that they are relatively easy to implement. As a drawback, it may happen that even if we have weak convergence for processes, 
the actual prices may not converge. For instance, the convergence of prices in exponential L\'{e}vy models is studied in 
details
by Cawston and Vostrikova \cite{Vostrikova}.

In this paper, we consider the convergence of prices of some sequence of simplified models. We introduce a 
notion of uniform integrability with respect to a bond-payoff pair $(B^n,f)$. Using this notion we show that if the 
prices for digital options converge, then the prices for a wide class of options converge if and only if the family of processes
is uniformly integrable with respect to a pair $(B^n,f)$ in given sequence of martingale measures $Q_n$. We also give more relaxed, but 
only sufficient conditions under which the prices converge.  

Under some additional conditions we also obtain exact error expansion for prices if we have such expansion 
for digital option prices. Moreover, this error expansion is relatively easy to compute from error expansion for digital 
options. As a benefit, we obtain an easily implemented method for pricing complicated claims. For these result we
use a pricing formulas derived by Viitasaari \cite{viitasaari} which can be applied for values of options at some time $t$, not only the
prices i.e. the values at time $t=0$. Hence, depending on the model and the corresponding approximation, it is possible
to construct approximations that converge faster to true values compared to some PDE-methods and are simpler to implement.

\subsection{Convergence of prices in binomial tree approximations}
We illustrate how to obtain error expansion for prices of general options by considering binomial tree approximation for
Black-Scholes prices. In a binomial model with $n$ time periods, the stock price $S$ either rises to $Su$ or falls to $Sd$ at each time period. 
With different choices of $u$ and $d$ one gets different kind of binomial models. 

It is well known that the binomial option prices converges to the Black-Scholes option prices as the number $n$ of time 
periods increase to infinity. The first proofs to this result were given by Cox et al. \cite{Cox} and Rendleman and 
Bartter \cite{Rendleman}. A more general proof was afterwards provided by Hsia \cite{Hsia}. 

Later the focus has been on the rate of convergence. Heston and Zhou \cite{Heston} worked with 
Cox, Ross, and Rubinstein (CRR) model, where $u=e^{\sigma\sqrt{\Delta t}}$ and $d=1/u$ with
$\sigma$ corresponding to the volatility of the stock in Black-Scholes model and $\Delta t=\frac{T}{n}$. They showed that 
the error is of order $O\left(\frac{1}{\sqrt{n}}\right)$ for general class of options. However, for European call options 
the convergence is faster. Leisen and Reimer \cite{Leisen} proved that for the CRR model, the Jarrow and Rudd \cite{Jarrow} 
model and the Tian \cite{Tian} model, the error for European call options is of order $O\left(n^{-1}\right)$. 
However, Leisen and Reimer did not give an explicit formula for the error. 

First two results on the exact formula for the error are due to Francine and Marc Diener \cite{Diener} and Walsh \cite{Walsh}. Francine and Marc Diener worked with the CRR model 
and proved the following result in our notation:

In the $n$-period CRR model, if $S_0=1$, the binomial price $V_{BT(n)}^C(K)$ of European call option with strike $K$ and maturity $T=1$ satisfies
$$
V_{BT(n)}^C(K)=V_{BS}^C(K) + \frac{e^{-\frac{d_1^2}{2}}}{24\sigma\sqrt{2\pi}}\frac{A-12\sigma^2(\Delta_n^2-1)}{n}+O\left(\frac{1}{n\sqrt{n}}\right)
$$
where $V_{BS}^C(K)$ is the corresponding Black-Scholes price, \\
$\Delta_n = 1-2frac\left[\frac{\log\left(\frac{1}{K}\right)+n\log d}{\log u -\log d}\right]$, $A=-\sigma^2(6+d_1^2 + d_2^2)+4(d_1^2-d_2^2)r-12r^2$, and $frac(x)$ denotes the fractional part of $x$.

Walsh considered general class of options. In particular, Walsh considered payoff functions which 
are piecewise $C^2$, regular, and the function and its derivatives are polynomially bounded. He 
chose $u=e^{\sigma\sqrt{\Delta t}+r\Delta t}$ and $d=e^{-\sigma\sqrt{\Delta t}+r\Delta t}$ in a binomial model with 
even number of periods. For call option he obtained
$$
V_{BT(n)}^C(K)=V_{BS}^C(K) + \frac{S_0e^{-\frac{d_1^2}{2}}}{24\sigma\sqrt{2\pi T}}\frac{A-12\sigma^2T(\Delta_n^2-1)}{n}+O\left(\frac{1}{n\sqrt{n}}\right),
$$
where $A$ is constant and 
$$
\Delta_n  = 1-2frac\left[\frac{\log\left(\frac{S_0}{K}\right)+n\log d}{\log u -\log d}\right].
$$

The $\Delta_n$ appearing in these formulas oscillates between $-1$ and $1$, and thus the binomial prices usually oscillate 
around Black-Scholes prices. As a result, it can happen that doubling the number of time steps leads to larger error. The 
smooth convergence, in which case one can improve the rate by standard extrapolation techniques, was studied by Chang and 
Palmer \cite{Palmer}. In particular, they generalised the results of Diener and Diener \cite{Diener} and Walsh \cite{Walsh} 
by considering a generalised class of binomial models with additional parameter $\lambda_n$ appearing in $u$ and $d$. As a 
result, they obtained an explicit formula for the error for digital and call options in their generalised class of binomial 
models, and with suitable choices of $\lambda_n$ they obtained smooth convergence. The result of Chang and Palmer is one of 
the main tools of our analysis, and we introduce their result in section \ref{sec_exact}.

We also consider generalised 
class of binomial models introduced by Chang and Palmer \cite{Palmer} and find an explicit formula for the error for 
general class of options. We consider payoff functions which are piecewise $C^1$ and polynomially 
bounded. Hence, our result covers wider
class of options than the result of Walsh. Moreover, our formula seems to be a simpler one and has a simpler proof. Our 
result is also valid if $n$ is odd number. We also consider smooth convergence and show that if we use a 
modification of binomial approximation 
rather than actual binomial approximation, we obtain smooth convergence for a class of options. 

The rate of convergence can also be improved by constructing the binomial model in suitable way. The problem was studied 
first by Joshi \cite{Joshi} for odd number of time periods and later by Xiao \cite{Xiao} for even number of time periods. 
However, in such cases one needs to consider the underlying option for which one wants to improve the convergence and the 
improved rate does not hold for general options. Combining results of \cite{Xiao} and \cite{Joshi} with our modified 
binomial approximation we can obtain faster convergence. However, we do not go into details since the binomial approximations 
are not the focus of this article.

The rest of the paper is organised as follows. In section \ref{sec_results} we state our main results with discussions. 
In section \ref{sec_exact} we consider exact error expansion for approximation error and apply our results to binomial 
tree approximation for Black-Scholes prices.
Section \ref{sec_proofs} is devoted to actual proofs.

\section{Convergence of prices in simplified model}
\label{sec_results}
\subsection{Auxiliary facts}
\begin{defn}
A market $(M)$ is represented by a six-tuple $(\Omega,\mathcal{F}, B, S, (\mathcal{F}_t),\P)$, where $B$ denotes the bond-process which is increasing, adapted and satisfies $B_0=1$. An underlying asset $X$ on that market is given by a process $X_t$ which is adapted to the filtration $\mathcal{F}_t$ generated by the stock process $S$.
\end{defn}
We assume that for a given model, the set of equivalent martingale measures is not empty and for chosen pricing measure $Q$ the price of a European type option $f(X_T)$ is denoted by $V_Q^f$ and given by
\begin{equation}
\label{price_def}
V_Q^f = \E_Q[B_T^{-1}f(X_T)].
\end{equation}
We denote the price of digital option $\textbf{1}_{X_T>a}$ by $V_Q^D(a)$ and the price of digital option $\textbf{1}_{X_T\geq a}$ by $\tilde{V}_Q^D(a)$. Similarly, we denote by $V_Q^C(a)$ the call price with strike $a$ given by
\begin{equation}
\label{mitta_call2}
V_Q^C(a) = \E_Q[B_T^{-1}(X_T-a)^+].
\end{equation}
We consider the following class of payoff functions:
\begin{defn}
For a function $f:\R_+\rightarrow\R$, we denote $f\in\Pi$ if $f$ is continuously differentiable except on a finite set of points
$0<s_1<s_2<\ldots<s_N$ (and possibly on $s_0=0$) in which $f$ and $f'$ have jump-discontinuities. The jump and the limits from the left at zero is defined as $\Delta_-f(0)=f(0-) =f'(0-)= 0$.
\end{defn}
\begin{rmk}
We could extend our analysis to cover a case when the amount of jumps is countable. However, in this case we would have to deal with infinite series corresponding to jumps and we would 
need some extra conditions related to jumps. Hence, for simplicity, we restrict our analysis to finite number of jumps. We also wish to emphasise that actually we may 
have that $f'$ has countable many discontinuities and still our results are valid exactly in the same form as long as $f$ has finite number of discontinuities.
\end{rmk}
\begin{defn}
For a function $f:\R_+\rightarrow\R$, we denote $f\in\Pi_Q(X_T)$ if the following conditions are satisfied:
\begin{enumerate}
\item
$f\in\Pi$, 
\item
$f(X_T)\in L^1(Q)$,
\item
$f$ satisfies
\begin{equation}
\label{cond1}
\lim_{x\rightarrow\infty}\lvert f(x-)\rvert Q(X_T \geq x) = 0
\end{equation}
and,
\item
the Riemann-Stieltjes integral
$$
\int_0^{\infty}f'(a)V_Q^C(\ud a)
$$
exists.
\end{enumerate}
\end{defn}
One of our main tools are the following pricing formulas. For proofs and details, we refer to \cite{viitasaari}.
\begin{thm}
\label{theor_v_general}
Let $f\in\Pi_Q(X_T)$. 
Then the price $V^f$ of an option $f(X_T)$ is given by
\begin{equation}
 \label{v_general2}
\begin{split}
 V^f &= f(0)\E_Q[B_T^{-1}] + \int_0^{\infty} f'(a)V_Q^D(a)\ud
a\\
&+ \sum_{k=0}^{N}\Delta_-f(s_k)\tilde{V}_Q^D(s_k)\\
&+ \sum_{k=0}^{N}\Delta_+f(s_k)V_Q^D(s_k),
\end{split}
\end{equation}
where 
$\Delta_-f(s_k) = f(s_k)-f(s_k-)$ and $\Delta_+f(s_k)$ is a jump on right side at $s_k$, respectively. 
\end{thm}
\begin{cor}
\label{kor_pricing_bick}
Let $f\in\Pi_Q(X_T)$. If $f'$ is absolutely continuous on every interval $(s_k,s_{k+1})$, then
\begin{equation}
\label{kaava_bick}
\begin{split}
 V^f &= f(0)\E_Q[B_T^{-1}] + \int_0^{\infty} f''(a)V_Q^C(a)\ud a\\
&+ \sum_{k=0}^{N}\Delta_-f(s_k)\tilde{V}_Q^D(s_k)\\
&+ \sum_{k=0}^{N}\Delta_+f(s_k)V_Q^D(s_k)\\
&+ \sum_{k=0}^N (f'(s_k+)-f'(s_k-))V_Q^C(s_k)
\end{split}
\end{equation}
\end{cor}

\subsection{Main theorems}
Assume that we have an arbitrage-free market model $(M)$. 
We assume that we have chosen, for some particular reason, one equivalent martingale measure $Q$ and we are 
given the underlying asset $X_T$. We emphasise that our model may not be complete, we simply choose one particular 
measure $Q$ among (possibly) many martingale measures. We consider a sequence of models 
$(\Omega^n,\mathcal{F}^n, B^n, S^n, (\mathcal{F}_t^n),\P^n)$ with underlying assets $X_T^n$ and measures $Q_n$ which 
converges to the original model in the following sense:
\begin{defn}
\label{def:model_convergence}
The model $(\Omega^n,\mathcal{F}^n, B^n, S^n, (\mathcal{F}_t^n),\P^n)$ with given martingale measures $Q_n$ converges to 
$(\Omega,\mathcal{F}, B, S, (\mathcal{F}_t),\P)$ with $Q$ if; 
\begin{enumerate}
\item
prices $V_{n,Q_n}^D(a)$ of a digital option $\textbf{1}_{X_T^n>a}$ under measure $Q_n$ and the price $V_Q^D(a)$ of a digital option $\textbf{1}_{X_T>a}$ under measure $Q$ satisfies
\begin{equation}
\label{def_digital_conv}
V_{n,Q_n}^D(a) = V_Q^D(a) + O(g(n)),
\end{equation}
for every continuity point of $V_Q^D(a)$ and for some function $g(n)$ such that $g(n)\rightarrow 0$ as $n$ tends to infinity.
\item
The bonds converge with the same rate i.e.
\begin{equation}
\label{bond_convergence}
\E_{Q_n}[(B_T^n)^{-1}] = \E_Q[B_T^{-1}] + O(g(n)).
\end{equation}
\end{enumerate}
\end{defn}
\begin{rmk}
Note that our definition is not far from classical notion of weak convergence. Indeed, if the bond is constant, 
then the convergence (\ref{def_digital_conv}) is the definition of weak convergence. However, 
for our purpose we need more general definition for convergence in order to also cover stochastic interest models. 
Note also that condition (2) is a technical condition which is not of interest. Indeed, it simply guarantees that when 
applying pricing formula (\ref{theor_v_general}) to compute the rate for prices, the first term does not make the 
rate worse. Note also that if the bonds do not convergence, we may still be able to construct an approximation for 
digital prices. As a result, we would obtain that call prices do convergence but put prices do not.
\end{rmk}
Once given the original model, the associated measure $Q$ and the underlying asset $X_T$ we refer to it with notation $(M)$. 
Similarly, we refer to the approximating sequence of models with $(M^n)$ and denote $(M^n)\rightarrow(M)$ if the model converge
in the sense of Definition \ref{def:model_convergence}. Price for an option $f(X_T^n)$ in simplified model are denoted by
$V_{n,Q_n}^f$. Similarly, prices for digital options $\textbf{1}_{X_T^n\geq a}$ and $\textbf{1}_{X_T^n >a}$ are denoted
 by $\tilde{V}_{n,Q_n}^D(a)$ and $V_{n,Q_n}^D(a)$.

Given measures $Q$ and $Q_n$ and the processes $X$ and $X_n$ we consider the following class of payoff functions: 
\begin{defn}
We denote $f\in\Pi(M,M^n)$ if the pricing assumptions are valid for models $(M)$ and $(M^n)$ for every $n$ i.e.
\begin{equation}
\Pi(M,M^n) = \Pi_{Q}(X_T)\cap\left(\bigcap_n \Pi_{Q_n}(X_T^n)\right),
\end{equation}
and $V_Q^D(a)$ and $f(a)$ have no common jump points. 
\end{defn}
\begin{rmk}
Note that everything now depends on the measures $Q_n$ and $Q$ and underlying processes $X_t^n$ and $X_t$. However, we omit this dependence on the notation to make the notation clearer.
\end{rmk}
We need one more definition before stating our main theorem.
\begin{defn}
Let $f\in\bigcap_n \Pi_{Q_n}(X_T^n)$. Then we say that $X_T^n$ is uniform integrable with respect to the pair $(B^n,f)$, if for every $\epsilon>0$ there exists a number $a$ such that
\begin{equation}
\sup_n \lvert\E_{Q_n}[(B_T^n)^{-1}f(X_T^n)\textbf{1}_{X_T^n > a}]\rvert < \epsilon.
\end{equation}
\end{defn}
For uniform integrability with respect to $(B^n,f)$, we use short notation and say that $X_T^n$ is $(B^n,f)$-UI.
\begin{thm}
\label{main_theorem}
Let $f\in\Pi(M,M^n)$ and assume that $(M^n)\rightarrow(M)$. Then the following are equivalent:
\begin{enumerate}
\item
$X_T^n$ is $(B^n,f)$-UI,
\item
there exists $\tilde{g}(n)\rightarrow 0$ as $n$ tends to infinity such that
\begin{equation}
\label{main_price_rate}
V_{n,Q_n}^f = V_Q^f + O(\tilde{g}(n)),
\end{equation}
\item
there exists $\tilde{g}(n)\rightarrow 0$ as $n$ tends to infinity such that
\begin{equation}
\label{integral_eq_cond}
\sup_n\left\lvert\int_0^\infty f'(a)(V_{n,Q_n}^D(a) - V^D(a))\tilde{g}(n)^{-1}\ud a\right\rvert < \infty
\end{equation}
\end{enumerate}
\end{thm}
\begin{rmk}
We restricted that $f$ has finite amount of jumps. If $f$ has countable many jumps, then in addition of 
(\ref{integral_eq_cond}) we also need similar condition for jumps.
\end{rmk}
\begin{rmk}
Note that the rate for prices of general options can be better or worse than for digital options. As a classical example 
omit the bond and consider a sequence of random variables given by $X_n = n^\alpha \textbf{1}_{\left[0,\frac{1}{n}\right]}$ 
for some $\alpha\leq 1$. Now we have $X_n \rightarrow 0$ almost surely and hence in distribution also. Moreover, the 
distribution functions converge with rate $\frac{1}{n}$. However, expectations does not converge at all for $\alpha=1$ and 
for $\alpha<1$, expectations converge at rate $n^{\alpha-1}$ which is worse than the rate for distribution functions. On the 
other hand, CRR-model serves as an example for an approximation when convergence for call prices is better than for digital 
prices. 
\end{rmk}
Our main theorem gives sufficient and necessity conditions on convergence for prices. However, the conditions may be difficult to check. The following corollary gives easier, but only sufficient conditions when we have the convergence of prices.
\begin{cor}
\label{main_theorem2}
Let $f\in\Pi(M,M^n)$ and assume that $(M^n)\rightarrow(M)$. Then
\begin{enumerate}
\item
if there exists $\tilde{g}(n)\rightarrow 0$ as $n$ tends to infinity such that
\begin{equation}
\label{sup_L1_function}
h(a):=\sup_n \lvert f'(a)(V_{n,Q_n}^D(a)-V^D_Q(a)) \tilde{g}(n)^{-1}\rvert\in L^1(\R_+),
\end{equation}
then conditions of Theorem \ref{main_theorem} are satisfied,
\item
if there exist a set $\Upsilon$ of Lebesgue-measure zero, a function $h(a)\in L^1(\R_+)$ and $\tilde{g}(n)\rightarrow 0$ as $n$ tends to infinity such that
\begin{equation}
\label{sup_sup_condition}
\sup_n\sup_{a\in \R_+- \Upsilon} \lvert [h(a)]^{-1}f'(a)(V_{n,Q_n}^D(a) - V_{Q}^D(a))\tilde{g}(n)^{-1}\rvert < \infty,
\end{equation}
then conditions of Theorem \ref{main_theorem} are satisfied.
\end{enumerate}
\end{cor}
\subsection{Notes on weak convergence and uniform integrability}
Let the bond $B_t=1$ for every $t$ and put $f(x)=x$. Then we can formulate our main theorem as follows:
\begin{thm}
Assume that $X_n$ converges to $X$ in distribution. Then the following are equivalent:
\begin{enumerate}
\item
$\E_{Q_n}[X_n] \rightarrow \E_Q[X]$,
\item
the family $X_n$ is uniformly integrable.
\end{enumerate}
\end{thm}
In the context of relation between convergence of random variables and uniform integrability, the following is well known. 
\begin{thm}
Let $X_n,X$ be positive random variables in $L^1$ and assume that $X_n \rightarrow X$ in probability. Then the following are equivalent:
\begin{enumerate}
\item
the family $X_n$ is uniformly integrable,
\item
$X_n \rightarrow X$ in $L^1$,
\item
$\E[X_n] \rightarrow \E[X]$.
\end{enumerate}
\end{thm}
According to our main theorem, we see directly that uniform integrability and convergence of expectations are equivalent under weak convergence, even if the random variables are defined on different probability spaces. There is also a nice characterisation of uniform integrability due to 
Leskel\"{a} and Vihola \cite{lasse}.  
\begin{defn}
A random variable $X$ is bounded in increasing convex order by $Y$ and denoted by $X\leq_{icx}Y$ if 
\begin{equation}
\E(X-a)^+ \leq \E(Y-a)^+\quad \forall a\in\R.
\end{equation}
\end{defn}
The authors in \cite{lasse} proved the following: 
\begin{thm}
Let $X_n$ be a family of positive random variables (possibly defined on different probability spaces). Then the following are equivalent: 
\begin{enumerate}
\item
the family $X_n$ is uniformly integrable,
\item
the family $X_n$ is icx-bounded by an integrable random variable $X$ i.e. $X_n\leq_{icx}X$ for every $n$,
\item
$
\lim_{t\rightarrow\infty}\sup_n \int_t^\infty \P_n(X_n > a)\ud a = 0.
$
\end{enumerate}
\end{thm}
It can be shown that if $X_n$ converges to an integrable random variable $X$ in distribution and $X_n$ is icx-bounded by $X$, then $X$ is the minimal element which dominates $X_n$ in increasing convex order in the sense that for any other dominating random variable $\tilde{X}$ we have $X\leq_{icx}\tilde{X}$. From financial point of view, this means that for call option our approximating price is too cheap compared to the option's true value.

We wish also to give a comment on weak convergence. It is well known (see \cite{jacod}, for instance) that the following are equivalent:
\begin{enumerate}
\item
$X_n\rightarrow X$ in distribution,
\item
$\E f(X_n)\rightarrow f(X)$ for all continuous bounded functions $f$,
\item
$\E f(X_n)\rightarrow f(X)$ for all functions $f\in C^\infty_0$.
\end{enumerate}
By Theorem \ref{main_theorem} we can deduce that actually the only important feature is that the first derivative of 
$f$ has compact support in order to have $(3)$. 
\section{Exact error expansion in simplified model}
\label{sec_exact}
Our main theorem tells us when the prices convergence. Following same technique, we can also obtain exact error expansion for general option prices if we have the exact error expansion for digital option prices. We clarify this with a guiding example where we have exact error coefficients for one leading term. Evidently, same method can be applied in a case we have many.

Assume we have exact error expansion for digital options and bonds of forms
\begin{equation}
V_{n,Q_n}^D(a) = V_Q^D(a) + A_n(a)g(n) + O(G(n)),
\end{equation}
\begin{equation}
\tilde{V}_{n,Q_n}^D(a) = \tilde{V}_Q^D(a) + \tilde{A}_n(a)g(n) + O(G(n)),
\end{equation}
and
\begin{equation}
\E_{Q_n}[(B_T^n)^{-1}] = \E_Q[B_T^{-1}] + B_n(a)g(n) + O(G(n)),
\end{equation}
where $g(n)\rightarrow 0$ as $n$ tends to infinity, $A_n(a)$ is bounded in $n$, and $G(n) = o(g(n))$. Assume now that $f\in\Pi(M,M^n)$ and satisfies 
\begin{equation}
\sup_n\left\lvert\int_0^\infty f'(a)A_n(a)\ud a\right\rvert < \infty.
\end{equation}
Then, computing formally, we have
\begin{equation}
\begin{split}
V_{n,Q_n}^f &= V_Q^f + g(n)\int_0^\infty f'(a)A_n(a)\ud a\\
&+g(n)\sum_{k=1}^N\Delta_-f(s_k)\tilde{A}_n(s_k)+g(n)\sum_{k=1}^N\Delta_+f(s_k)A_n(s_k)\\
&+f(0)B_n(a)g(n) + G(n)\int_0^\infty f'(a)O_n(a)\ud a\\
&+O(G(n))
\end{split}
\end{equation}
for some function $O_n(a)$ which is bounded in $n$. In order to get an exact error expansion for prices, we would like to show that
\begin{equation}
\label{last_term_hope}
G(n)\int_0^\infty f'(a)O_n(a)\ud a = O(G(n)).
\end{equation}
However, it is not complicated to construct counter-examples to show that this is not true in general. For instance, 
put $O_n(a) = \textbf{1}_{n>a}$, $G(n)=n^{-1}$, and $f(x)=x$. 

In order to have (\ref{last_term_hope}) we can apply Theorem \ref{main_theorem} and Corollary \ref{main_theorem2} 
for a modified digital approximation given by
\begin{equation}
\label{digital_appro_modified}
\hat{V}_{n,Q_n}^D(a) = V_{n,Q_n}^D(a) - A_n(a)g(n)
\end{equation}
and if the rate turns out to be suitable then we can, for example have error expansion of the form (here we have omitted
the effect of bond and jumps)
\begin{equation}
V_{n,Q_n}^f = V_Q^f + g(n)\int_0^\infty f'(a)A_n(a)\ud a + O(G(n)).
\end{equation}
However, typically in an error expansion we are only interested in exact coefficients and the last term is negligible such 
that we only want to know that it is very small compared to other terms. For instance, we could express the error expansion 
in form
\begin{equation}
V_{n,Q_n}^f = V_Q^f + g(n)\int_0^\infty f'(a)A_n(a)\ud a + o(g(n)).
\end{equation}
Instead of applying our main theorems, this kind of reduction can be done with easier conditions to check. For simplicity, we state the result in a simple form. For extensions, see discussion and remarks below. 
\begin{thm}
\label{thm:exact_error}
Let $f\in\Pi(M,M^n)$ be continuous. Assume that $\E_{Q_n}[(B_T^n)^{-1}]$ converges to $\E_Q[B_T^{-1}]$ with 
rate $o(n^{-\beta})$ and digital option prices satisfy
\begin{equation}
V_{n,Q_n}^D(a) = V_Q^D(a) + \frac{A_n(a)}{n^\beta} + O(n^{-\kappa}),
\end{equation}
where $A_n(a)$ is bounded in $n$.
Put 
\begin{equation}
F_{\delta,\theta}(a) = \int_0^a \lvert y^{\delta}f'(y)\rvert^{1+\theta}\ud y.
\end{equation}
If there exist constants $\delta$ and $\theta$ such that
\begin{equation}
\label{cond_eps_delta}
\theta > \frac{\beta}{\kappa-\beta}, \quad \delta > \frac{\theta}{1+\theta},
\end{equation}
\begin{equation}
\label{finite_limit_price}
V_Q^{F_{\delta,\theta}} < \infty,
\end{equation}
\begin{equation}
\label{bounded_coef2}
\sup_n\int_0^\infty \lvert a^\delta f'(a)\rvert^{1+\theta}\lvert A_n(a)\rvert\ud a < \infty,
\end{equation}
and $X_T^n$ is $(B^n,F_{\delta,\theta})$-UI, then
\begin{equation}
\label{exact_error_expansion}
V_{n,Q_n}^f = V_Q^f + \frac{1}{n^\beta}\int_0^\infty f'(a)A_n(a)\ud a + o(n^{-\beta}).
\end{equation}
\end{thm}
We assumed that the payoff function $f$ is continuous. However, since the amount of jumps is finite we simply need to have 
error expansion for digital prices $\tilde{V}_{n,Q_n}^D(a)$ and we only have to add the terms corresponding to jumps 
in order to have expansion for discontinuous functions. Note also that we introduced the theorem in a simple form where 
we only had one leading error term, the bonds does not have effect and the expansion was in powers of $n$. This was due 
to the fact that the main 
contribution of theorem is to explain how integrals of functions $O(n^{-\kappa})$ can be dealt easily in our 
situation. Evidently we could extend similar analysis to cover a case where the error expansion contains more terms or 
the rate is given by some function $G(n)$ which is not of power type. Moreover, if the error expansion for bonds is not 
of given form, it only affects by terms $f(0)(\E_{Q_n}[(B_T^n)^{-1}] - \E_Q[B_T^{-1}])$ and the final error expansion is easy to 
modify from (\ref{exact_error_expansion}). 
\begin{rmk}
Note that (\ref{bounded_coef2}) implies
\begin{equation}
\label{exact_error_bounded_coefficient}
\sup_n\left\lvert\int_0^\infty f'(a)A_n(a)\ud a\right\rvert < \infty,
\end{equation} 
and hence (\ref{exact_error_expansion}) is reasonable.
\end{rmk}
\begin{rmk}
Technical assumptions on Theorem \ref{thm:exact_error} may seem a bit odd at first sight and hence we wish to give financial interpretation. From financial point of view, we are especially interested in having convergence for call prices which is equivalent to considering convergence for discounted expectations. For simplicity, let us omit the bond and assume that $f'=1$. We are interested, for instance, to find an error expansion of the form
$$
V_{n,Q_n}^f = V_Q^f + \frac{A_n(a)}{n}+o(n^{-1}).
$$
Let
$$
\rho = \sup\{p : \max[\sup_n\E_{Q_n}[X_T^n]^p,\E_Q[X_T]^p]<\infty\}.
$$
From financial point of view, the case $\rho<1$ is not of interest and hence we assume that $\rho\geq 1$. Now it is straightforward to see that moments converge for every $p<\rho$, and hence we have to choose $\theta < \rho -1$. On the other hand, we need that 
$$
\frac{\kappa\theta}{1+\theta} > 1 \Leftrightarrow \theta > \frac{1}{\kappa-1}.
$$
This implies that we have to compute the error expansion for digital options up to term 
$$
\kappa > \frac{\rho}{\rho-1}.
$$
In other words, if the underlying assets have fewer moments then we need to compute more accurate error expansion for digital option prices.
\end{rmk}
\begin{rmk}
Let the last term $O(n^{-\kappa})$ be given by $\frac{O_n(a)}{n^\kappa}$ for some function $O_n(a)$ which is bounded in $n$. If $f'$ has compact support, then we always have that
$$
\int_0^\infty f'(a)\frac{O_n(a)}{n^\kappa}\ud a = O(n^{-\kappa}).
$$
If $V_{n,Q_n}^f$ converges to $V_Q^f$ with rate $O(n^{-\beta})$ and we have 
(\ref{exact_error_bounded_coefficient}), then we also have
$$
\sup_n\left\lvert\int_0^\infty f'(a)\frac{O_n(a)}{n^{\kappa-\beta}}\ud a\right\rvert < \infty.
$$
Hence, if we have the following uniform integrability: $\forall\epsilon>0$ there exists $K$ s.t.
$$
\left\lvert\int_K^\infty f'(a)\frac{O_n(a)}{n^{\kappa-\beta}}\ud a\right\rvert < \epsilon,
$$
we also have
$$
V_{n,Q_n}^f = V_Q^f + \frac{1}{n^\beta}\int_0^\infty f'(a)A_n(a)\ud a + o(n^{-\beta}).
$$
\end{rmk}
\begin{rmk}
If we know the exact error expansion for call prices, we can, 
if $f$ has enough smoothness, apply equation (\ref{kaava_bick}) to obtain exact error expansion. This can be useful indeed, 
since the convergence can be faster for call options. Evidently, the error expansion does not depend on the particular 
choice between equations (\ref{v_general2}) and (\ref{kaava_bick}) but it may be easier and more clear to consider equation 
(\ref{kaava_bick}) instead of (\ref{v_general2}). As an example, this becomes particularly clear for binomial models.  
\end{rmk} 
According to Theorem \ref{thm:exact_error} it is sufficient to study the error for digital and call options only and even 
the exact error expansion follows, which can be particularly useful. In next section we illustrate the usage of Theorem \ref{thm:exact_error} 
by considering generalised binomial tree approximation of Black-Scholes model, for which, thanks to Chang and Palmer 
\cite{Palmer}, the error expansions for call and digital prices are explicitly known. 
\subsection{Error expansion for binomial tree prices}
\label{binomial}
In this section we consider binomial tree approximations for Black-Scholes prices. We restrict our analysis to stock options i.e. we consider $X_T^n = S_T^n$ and $X_T = S_T$, where $S_t$ follows a geometric Brownian motion. We use the following notation: $\sigma$ as the volatility,
$r$ as the continuously compounded interest rate and $T$ as the maturity. We denote by $V_{BT(n)}^C(a)$ the price of the call option with strike $a$ and maturity $T$ in the $n$-period binomial model and $V_{BS}^C(a)$ the Black-Scholes price of the call option with strike $a$ and maturity $T$, given by:
$$
V_{BS}^C(a) = S_0\Phi(d_1) - ae^{-rT}\Phi(d_2),
$$
where $\Phi(\cdot)$ is the standard normal distribution function and $d_i$ are given in Appendix A. Similarly, the prices for digital options 
$\textbf{1}_{x\geq a}$ with strike $a$ are denoted by $\tilde{V}_{BT(n)}^D(a)$ and $\tilde{V}_{BS}^D(a)$. 
Prices for option $f(S_T)$ are denoted by $V_{BT(n)}^f$ and $V_{BS}^f$. We use notation $f\in\Pi(BS,BT(n))$ to emphasise that we consider Black-Scholes model and binomial 
tree models with $n$ time steps.
We also need several short notations for different functions of $n$ and $a$. The list and definitions are given in appendix.

We consider the following generalised class of binomial models introduced by Chang and Palmer \cite{Palmer}.
\begin{defn}[The class of binomial models]
Let $\Delta t=\frac{T}{n}$ and $\lambda_n$ an arbitrary bounded function of $n$. We consider
$n$-period binomial model, where
$$
u=e^{\sigma\sqrt{\Delta t}+\lambda_n\sigma^2\Delta t}, \quad d =e^{-\sigma\sqrt{\Delta
t}+\lambda_n\sigma^2\Delta t},
$$
with initial stock price $S_0$. 
\end{defn}
With different choices of $\lambda_n$ one gets different binomial models. For example, the choice $\lambda_n=0$ gives the Cox, Ross
and Rubinstein model, and the
choice $\lambda_n = r/\sigma^2 - 1/2$ gives the Jarrow and Rudd model \cite{Jarrow}. 
The following error expansions in a generalised binomial model for digital and call options was proved in \cite{Palmer}.
\begin{thm}
\label{theor_rate_call}
Let $\Delta t=\frac{T}{n}$ and $\lambda_n$ an arbitrary bounded function of $n$. For the
$n$-period binomial model, where
$$
u=e^{\sigma\sqrt{\Delta t}+\lambda_n\sigma^2\Delta t}, \quad d =e^{-\sigma\sqrt{\Delta
t}+\lambda_n\sigma^2\Delta t},
$$
if the initial stock price is $S_0$ and the strike price is $a$ and maturity $T$, then
\begin{enumerate}
 \item
the price of a digital call option satisfies
\begin{equation}
\label{error_digital_call}
\begin{split}
\tilde{V}^D_{BT(n)}(a) &= \tilde{V}^D_{BS}(a) +
\frac{e^{-rT}e^{-\frac{-d_2^2}{2}}}{\sqrt{2\pi}}\left[\frac{\Delta_n}{\sqrt{n}} -
\frac{d_2(\Delta_n)^2}{2n} + \frac{B_n}{n}\right]\\ 
&+ O\left(\frac{1}{n\sqrt{n}}\right),
\end{split}
\end{equation}
and
\item
the price of a European call option satisfies
\begin{equation}
\label{error_call}
\begin{split}
V^C_{BT(n)}(a) &= V^C_{BS}(a) + \frac{S_0e^{-\frac{-d_1^2}{2}}}{24\sigma\sqrt{2\pi
T}}\frac{A_n - 12\sigma^2T((\Delta_n)^2-1)}{n}\\
&+ O\left(\frac{1}{n\sqrt{n}}\right),
\end{split}
\end{equation}
\end{enumerate}
where $\Delta_n$, $A_n$ and $B_n$ are given in Appendix $A$. 
\end{thm}
\begin{rmk}
In \cite{Palmer}, the authors actually stated their result in a form where the remaining term was given by 
$o(n^{-1})$. However, by careful examination of the proof one obtains that it is actually 
$O\left(\frac{1}{n\sqrt{n}}\right)$.
\end{rmk}
\begin{rmk}
In the proof of Theorem \ref{theor_rate_call}, the authors
assumed that the strike satisfies $S_0d^n\leq a \leq \frac{u}{d}S_0u^n$. This is a natural
assumption since in the $n$-period binomial model, the price
of a digital option is $e^{-rT}$ and the price of a call option is $S_0$ if $a<S_0d^n$. Similarly, for
$a>S_0u^n$ the prices for both are zero.
However, it is straightforward to see that the formulas of Theorem \ref{theor_rate_call}
are true for any strike price.
\end{rmk}
\begin{rmk}
\label{rate_rema}
In \cite{Palmer}, the authors used payoff function $f(x)=\textbf{1}_{x\geq a}$ for
digital option with strike $a$. Evidently, if the strike is not on a lattice point, 
the formula (\ref{error_digital_call}) holds for the payoff
$f(x)=\textbf{1}_{x> a}$ as well. For our result, we need similar formula for payoff
$f(x)=\textbf{1}_{x> a}$ when the strike is on a lattice point. By examining the proof of Chang
and Palmer \cite{Palmer}, we obtain that
in this case, the formula is given by (\ref{error_digital_call}), where
$\Delta_n$ is replaced with $\Delta_n(a)-2$. 
In other words, if $a$ is a terminal stock price we have
\begin{equation}
\label{error_digital_call2}
\begin{split}
V^D_{BT(n)}(a) &= V^D_{BS}(a) +
\frac{e^{-rT}e^{-\frac{-d_2^2}{2}}}{\sqrt{2\pi}}\left[\frac{\Delta_n-2}{\sqrt{n}} -
\frac{d_2(\Delta_n-2)^2}{2n} + \frac{B_n}{n}\right]\\ 
&+ O\left(\frac{1}{n\sqrt{n}}\right),
\end{split}
\end{equation}
For the details, see
\cite{Palmer}.
\end{rmk}
\begin{lma}
\label{lma:ok_pricing}
Let $f\in\Pi$ and assume that $f'$ is polynomially bounded i.e. there exists positive constants $c_1,c_2$ and $p$ such that $\lvert f'(a)\rvert\leq c_1 + c_2a^p$ for every 
$a\geq 0$. Then $f\in \Pi(BS,BT(n))$.
\end{lma}
Now we can state our main theorem of this section:
\begin{thm}
\label{theor_rate_general}
Let $f\in\Pi$ and assume that $f'$ is polynomially bounded.
Then for the $n$-period generalised binomial model, the price of a European option $f(S_T)$ satisfies
\begin{equation}
 \label{error_general}
\begin{split}
V^f_{BT(n)} &= V^f_{BS} + \frac{1}{\sqrt{n}}\int_0^\infty
f'(a) \frac{e^{-rT}e^{-\frac{d_2^2}{2}}}{\sqrt{2\pi}}\Delta_n\ud a \\
&+ \frac{1}{n}\int_0^\infty f'(a)\frac{e^{-rT}e^{-\frac{d_2^2}{2}}}{\sqrt{2\pi}}\left[B_n -\frac{d_2\Delta_n^2}{2}\right]\ud a\\
&+ \sum_{N_1(s_k)}\Delta_-f(s_k)J_n(s_k)+ \sum_{N_1(s_k)}\Delta_+f(s_k)\widehat{J}_n(s_k)\\
&+ \sum_{N_2(s_k)}\left[\Delta_-f(s_k)+\Delta_+f(s_k)\right]J_n(s_k)\\
&+o\left(\frac{1}{n}\right),
\end{split}
\end{equation}
where the functions $\Delta_n$, $B_n$, $J_n$ and $\widehat{J}_n$ are defined in Appendix A, $N_1(s_k)$ is the set of jumps of $f$ on the lattice points and $N_2(s_k)$ is the set of jumps of $f$ not on the lattice points.
Moreover, the integrals are bounded in $n$. 
\end{thm}
\begin{cor}
\label{theor_rate_c2}
Let $f\in\Pi$ such that $f'(a)$ is absolutely continuous on every interval $(s_k,s_{k+1})$. Moreover, assume that 
$f'$ and $f''$ are polynomially bounded. 
Then for the $n$-period generalised binomial model, the price of a European option $f(S_T)$ satisfies
\begin{equation}
 \label{error_general_c2}
\begin{split}
V^f_{BT(n)} &= V^f_{BS} + \frac{1}{n}\int_0^\infty
f''(a)H_n(a)\ud a \\
&+ \sum_{N_1(s_k)}\Delta_-f(s_k)J_n(s_k)+ \sum_{N_1(s_k)}\Delta_+f(s_k)\widehat{J}_n(s_k)\\
&+ \sum_{N_2(s_k)}\left[\Delta_-f(s_k)+\Delta_+f(s_k)\right]J_n(s_k)\\
&+ \frac{1}{n}\sum_{k=0}^N (f'(s_k+)-f'(s_k-))H_n(s_k)\\
&+o\left(\frac{1}{n}\right),
\end{split}
\end{equation}
where the functions $H_n$, $J_n$ and $\widehat{J}_n$ are defined in Appendix A, $N_1(s_k)$ is the set of jumps of $f$ on the lattice points and $N_2(s_k)$ is the set of jumps of $f$ not on the lattice points.
Moreover, the integral is bounded in $n$. 
\end{cor}
Our theorems show that for a wide class of options,
the rate of convergence is of order $O(\frac{1}{\sqrt{n}})$. Moreover, the derivative $f'$ of the payoff
function is usually absolutely continuous and hence the second derivative exists almost everywhere.
In this case, the rate is of order $O(n^{-1})$ and if the
function has discontinuities, then the rate is of order
$O(1/\sqrt{n})$. 

Walsh also gave an explicit formula for the error and proved that under his assumptions, the rate of convergence is of
order $O(1/\sqrt{n})$ if the discontinuity is not on a lattice point and the rate of convergence is of order $O(n^{-1})$ 
if all the discontinuities are on lattice points. Compared to the formula of Walsh, our result covers a wider class of 
payoff functions and a wide class of binomial models. Note also that we obtain, as Walsh did under his assumptions, 
that if $f'$ is absolutely continuous on every interval $(s_k,s_{k+1})$, all the 
discontinuities lie on lattice points and the function is regular, then 
the rate of convergence is of order $O(n^{-1})$ for any binomial model within our class. Indeed, if all the discontinuities
lie on lattice points, then $\Delta_n(s_k)=1$ for all jump points $s_k$ and by regularity, $\Delta_-f(s_k)=\Delta_+f(s_k)$.
Thus the coefficients of the term $\frac{1}{\sqrt{n}}$ cancel and we obtain convergence of order $O(n^{-1})$.

\subsection{Smooth convergence in binomial tree approximation and improving speed of convergence}
Typically binomial prices oscillates around Black-Scholes prices, and for general options this is indeed so as can be seen from Theorem 
\ref{theor_rate_general}. This is unwanted feature, since when the prices oscillates around Black-Scholes price, it may happen that by doubling the amount of steps $n$ one also double the error.
The main reason why the authors in \cite{Palmer} considered generalised class of binomial models was the smooth convergence, formal definition given below, for which this oscillation feature no longer holds.
\begin{defn}
We say that $a_n$ converges smoothly to $a$ with order $\beta$, if
\begin{equation}
a_n = a + \frac{C}{n^\beta} + o(n^{-\beta})
\end{equation}
for some constant $C$. 
\end{defn}
When the convergence is smooth, it follows that for large $n$ the sequence $a_n$ is monotone and hence there is no oscillation. In this case, we can use standard extrapolation techniques to accelerate the convergence. For more details on smooth convergence of binomial prices, 
we refer to \cite{Tian} and \cite{Palmer}.

In \cite{Palmer}, the authors proved that with suitable choice of $\lambda_n$ in generalised binomial model, one can obtain smooth convergence for digital and call options. More precisely, they showed that if $\lambda_n$ is chosen such that 
the strike coincides with a stock price at terminal node, then the convergence for digital options is smooth with order $\frac{1}{2}$ and the convergence for call options is smooth with order $1$. More interesting, Chang and Palmer proposed a center binomial model, where $\lambda_n$ is chosen such that 
the strike is a geometric average of two terminal stock price. In this center binomial model, they obtained smooth convergence of order $1$ for both, digital and call option.
In particular, for center binomial model they chose
\begin{equation}
\label{lambda_centered}
\lambda_n(a) = \frac{\log\left(\frac{a}{S_0}\right) - (2j_0 - 1 - n)\sigma\sqrt{\Delta t}}{n\sigma^2\Delta t},
\end{equation}
where 
\begin{equation}
\tilde{\gamma} = \frac{\log\left(\frac{a}{S_0}\right) + n\sigma\sqrt{\Delta t}}{2\sigma^2\Delta t},
\end{equation}
\begin{equation}
j_0 = \min\{m\in\N : m \geq \tilde{\gamma}\},
\end{equation}
and $a$ is the strike of the option. 
Since for smooth convergence the choice of $\lambda_n$ depends on the strike, we cannot obtain smooth convergence in general 
for every option $f$. However, if we do not use direct binomial approximation but instead combine different center 
binomial models we obtain smooth convergence for more general payoff functions $f$. In principal, we could simply approximate 
Black-Scholes price by computing center binomial digital price for jump terms and integral
\begin{equation}
\label{smooth_principal}
\int_0^\infty f'(a)V_{n,\lambda(a)}^D(a)\ud a,
\end{equation}
where $V_{n,\lambda(a)}^D(a)$ is the price of a digital option in center binomial model with strike $a$ and $\lambda(a)$ 
given by (\ref{lambda_centered}). From implementation point of view however, we can only approximate integral 
(\ref{smooth_principal}) 
with suitable numerical integration method. We show how to obtain smooth convergence 
only for functions which has enough smoothness, but more general cases can be obtained by obvious modifications 
(see Remark \ref{rema_smooth_extension}).

We consider an interval $[0,n^\alpha]$ for some $\alpha>0$ and trapezoidal approximation with $n$ equidistant points for integrals of type
$$
\int_0^{n^\alpha}g(a)\ud a.
$$
The approximation is given by 
\begin{equation}
\label{trapezoid_appro}
I_{t,n} = \frac{n^{\alpha-1}}{2}\sum_{k=0}^{n-1}\left(g(kn^{\alpha-1})+ g((k+1)n^{\alpha-1})\right).
\end{equation}
\begin{thm}
\label{thm:smooth}
Let $f$ be three times continuously differentiable such that all derivatives of $f$ are polynomially bounded. 
Let $\lambda_n(a)$ be given by (\ref{lambda_centered}) and let $V_{n,\lambda(a)}^D(a)$ denote 
the price of a digital option with strike $a$ in center binomial model with $\lambda_n(a)$. If $\alpha<\frac{1}{3}$, then
\begin{equation}
\label{smooth_trapet}
\begin{split}
V_{t,n}^f &= \frac{n^{\alpha-1}}{2}\sum_{k=0}^{n-1}f'(kn^{\alpha-1})V_{n,\lambda(kn^{\alpha-1})}^D(kn^{\alpha-1})\\
&+ \frac{n^{\alpha-1}}{2}\sum_{k=0}^{n-1}f'((k+1)n^{\alpha-1})V_{n,\lambda((k+1)n^{\alpha-1})}^D((k+1)n^{\alpha-1})
\end{split}
\end{equation}
converges smoothly to Black-Scholes price with order $\frac{1}{n}$ i.e. 
\begin{equation}
V_{t,n}^f = V_{BS}^f + \frac{C}{n} + o\left(\frac{1}{n}\right),
\end{equation}
where
\begin{equation}
\label{smooth_constant}
C = \int f'(a)C(a)\ud a,
\end{equation}
\begin{equation}
\label{smooth_const_function}
\begin{split}
C(a) &= \frac{e^{-rT}e^{\frac{d_2^2}{2}}}{\sqrt{2\pi}}\cdot\frac{d_1^3+d_1d_2^2+2d_2-4d_1}{24}\\
&+\frac{e^{-rT}e^{\frac{d_2^2}{2}}}{\sqrt{2\pi}}\cdot\left[\frac{(2-d_1d_2-d_1^2)\sqrt{T}}{6\sigma}r + \frac{Td_1r^2}{2\sigma^2}\right].
\end{split}
\end{equation}
\end{thm}
\begin{exm}
As a non-trivial example, consider power option $f(x)=(x^n - K)^+$. In this case we have $f(x)=0$, if $x\leq K$. Hence we have to approximate integral over interval $(K,n^{\alpha})$ and on this set $f$ is smooth enough to apply Theorem \ref{thm:smooth}.
\end{exm}
\begin{rmk}
\label{rema_smooth_extension}
For simplicity we assumed that $f$ is continuous. However, if $f$ has jumps, then we simply add terms $\Delta_+f(s_k)V_{n,\lambda(s_k)}^D(s_k)$ and 
$\Delta_-f(s_k)V_{n,\lambda(s_k)}^D(s_k)$ corresponding to the jumps into our estimators $V_{q,n}^f$ and $V_{t,n}^f$. Similarly, if $f$ is only piecewise smooth   we simply compute integrals separately over different intervals. Note also that we only considered trapezoidal rule. Depending on the smoothness of $f$ however, we can apply any numerical integration method we wish. Hence, though the rate of convergence remains the same, one could use algorithms that are more sophisticated in order to decrease the error.
\end{rmk}
Our approximation (\ref{smooth_trapet}) is not binomial approximations. 
In both approximations we compute $n$ different binomial probabilities, but for every point $a_k$ the corresponding 
binomial probability is from different binomial distribution. However, from the implementation point of view our 
approximation is extremely easy to implement, as we only adjust $\lambda_n$ according to grid point $a_k$, compute the 
corresponding binomial probability and compute weighted sum of all points $a_k$. As a benefit, we obtain smooth convergence, 
and hence we may apply standard extrapolation techniques such as Richardson extrapolation. 
With same procedure we can improve the speed of convergence for continuous options $f$ with sufficient smoothness. 
In this case we can use formula \ref{v_general2} and approximate the integral 
$$
\int_0^\infty f''(a)V_{BS}^C(a)\ud a
$$
numerically. Then, if the number of steps $n$ is even, we can compute the call prices for every $a_k$ in the model proposed 
by Xiao \cite{Xiao} and obtain better rate. Similarly, if $n$ is odd, we can follow Joshi \cite{Joshi} to obtain faster convergence. 
For more details about improving the rate of convergence, the interested reader is referred to original papers.
\section{Conclusions}
\label{sec_conclusions}
From financial point of view, we obtain that the prices converge for a wide class of options. In particular, usually payoff 
functions are polynomially bounded. For such options, according to Theorem \ref{main_theorem}, prices converge for many 
distributions of interest of the underlying process $X$. For instance, our theorem is applicable for many models such as 
stochastic interest rate and volatility models, and jump-diffusion models or other exponential Levy models
(for detailed study of convergence of prices in Levy models, see \cite{Vostrikova}). More precisely, 
for polynomially bounded payoff functions $f$, the convergence of prices depends on how many moments of $X$ are finite. We also
emphasise that our result cover wide class of exotic options, but not Barrier options. However, 
in \cite{viitasaari} the author proved similar pricing formulas for Barrier type options, and hence our results could 
be formulated for these kind of options as well.

Our results can also be applied when considering the values of options at some time $t$. 
Another method for option valuation is to solve the corresponding PDE or PIDE. Evidently, 
solutions for these problems are rarely available in closed form and one has to use numerical methods such as finite difference methods. 
In such cases one has to consider different important concepts such as stability and accuracy. For example, some of the 
most simple difference methods for Black-Scholes PDE have error of order $O(\Delta t) + O(\Delta x^2)$. That is, the error 
is proportional to the time step and the square of the space step. 

According to Theorem \ref{thm:exact_error} it can be more reasonable in some situations to approximate the distribution function of the 
underlying process $X_t$ (or more precisely, the price of a digital option) with some sequence of processes $\{X_t^n\}_{n\geq 1}$ 
such that the value of the option $f(X_T^n)$ can be computed, exactly or with good accuracy, in the approximating model. 
In particular, this can be extremely useful if one can obtain sufficiently good rate for conditional distribution function. 
As a result one can obtain fast algorithms for option valuation by finding approximations such that digital option prices 
convergence sufficiently fast, and on approximating model one can find value for option $f(X_T^n)$ exactly or with a proper 
algorithm. A derivation of such approximations for different processes 
$X_t$ and models $(\Omega,\mathcal{F}, B, S, (\mathcal{F}_t),\P)$ would be an interesting subject of further research. 

Note also that as difference methods solve equation recursively, the error with respect to time discretisation can cumulate. 
Using our method the error in time does not cumulate if the prices for digital options can be valued for every $t$. 
We also emphasise that using our method the amount of computer work is not necessarily larger than using methods for solving 
PDEs. For example, for finite difference method one usually needs to solve $n$-dimensional system of linear equations 
for every time step and this has to be done separately for every option $f$. Compared to our method, we need to compute a 
product of known matrix and a vector that arise from numerical integration and discretisation in space. Although, our 
matrix is full while the matrices arising from difference methods are sparse and thus admit specifically adapted algorithms 
for solving the linear equation system. In our method, the 
computation of value in the 
approximating model should be computationally easy.
Note also that once one has access to digital prices, our method is extremely easy to implement compared to many 
sophisticated finite difference algorithms.

\section{Proofs}
\label{sec_proofs}
We begin with an auxiliary lemma.
\begin{lma}
\label{lma:perus}
Let $Y\in L^1$. Then for every $\epsilon>0$ there exists $\delta>0$ such that
$\P(A)<\delta\Rightarrow \E[Y\textbf{1}_A]<\epsilon$.
\end{lma}
\begin{proof}[Proof of Theorem \ref{main_theorem}]
Without loss of generality we can assume that $f$ is continuous and $f(0)=0$, since the amount of jumps is finite and 
convergence 
$$
f(0)\E_{Q_n}[(B_T^n)^{-1}] \rightarrow f(0)\E_Q[B_T^{-1}]
$$
follows from (\ref{bond_convergence}).\\
$(2)\Rightarrow(3)$: By (\ref{v_general2}) and $(2)$ we have
\begin{equation*}
\begin{split}
&\lvert V_{n,Q_n}^f - V_Q^f \rvert \\
&= \lvert \int_0^\infty f'(a)(V_{n,Q_n}^D(a)-V_Q^D(a))\ud a\rvert \\
&\leq C\tilde{g}(n)
\end{split}
\end{equation*}
as $n$ tends to infinity. Hence we have $(3)$.\\
$(3)\Rightarrow(2)$: Now we have
\begin{equation*}
\begin{split}
&\lvert V_{n,Q_n}^f - V_Q^f \rvert \\
&= \lvert \int_0^\infty f'(a)(V_{n,Q_n}^D(a)-V_Q^D(a))\ud a\rvert \\
&\leq \tilde{g}(n)\sup_n \lvert \int_0^\infty f'(a)(V_{n,Q_n}^D(a)-V_Q^D(a))\tilde{g}(n)^{-1}\ud a\rvert.
\end{split}
\end{equation*}
Hence we have $(2)$.\\
$(2)\Rightarrow(1)$: Let $\epsilon>0$ be fixed. Assume that $X_T^n$ is not 
$(B^n,f)$-UI and we have $(2)$. From $(2)$ we conclude that there exists $N$ such that
$$
\lvert\E_{Q_n}[(B_T^n)^{-1}f(X_T^n)\textbf{1}_{X_T^n>a}] - \E_Q[B_T^{-1}f(X)\textbf{1}_{X_T>a}]\rvert < \frac{\epsilon}{2},
$$
when $n\geq N$. By Lemma \ref{lma:perus} we can take $a$ large enough such that 
$$
\lvert\E_Q[B_T^{-1}f(X_T)\textbf{1}_{X_T>a}]\rvert<\frac{\epsilon}{2}
$$ 
which implies that we have 
$$
\lvert\E_{Q_n}[(B_T^n)^{-1}f(X_T^n)\textbf{1}_{X_T^n>a}]\rvert<\epsilon.
$$ 
Now $\epsilon$ is arbitrary, which 
implies that for every $\epsilon>0$ we can find a number $a$ independent of $n$ such that 
$\lvert\E_{Q_n}[(B_T^n)^{-1}f(X_T^n)\textbf{1}_{X_T^n>a}]\rvert < \epsilon$. Hence we have a contradiction.\\
$(1)\Rightarrow(2)$: It is evident that for every continuity point $a$ of $V_Q^D(a)$ we have
$$
\E_{Q_n}[(B_T^n)^{-1}f(X_T^n)\textbf{1}_{X_T^n\leq a}] \rightarrow \E_Q[B_T^{-1}f(X_T)\textbf{1}_{X_T\leq a}].
$$
Let now $\epsilon>0$ be fixed. By $(1)$ and Lemma \ref{lma:perus} we can take a continuity point $a$ of $V_Q^D(a)$ such that
$$
\sup_n \lvert\E_{Q_n}[(B_T^n)^{-1}f(X_T^n)\textbf{1}_{X_T^n>a}]\rvert < \frac{\epsilon}{3}
$$
and
$$
\lvert\E_Q[B_T^{-1}f(X_T)\textbf{1}_{X_T>a}]\rvert < \frac{\epsilon}{3}.
$$
Finally, take $N$ such that 
$$
\lvert\E_{Q_n}[(B_T^n)^{-1}f(X_T^n)\textbf{1}_{X_T^n\leq a}] - \E_Q[B_T^{-1}f(X)\textbf{1}_{X_T\leq a}]\rvert < \frac{\epsilon}{3}
$$
for every $n\geq N$. This implies that $V_{n,Q_n}^f\rightarrow V_Q^f$ i.e. we have $(2)$.
\end{proof}
\begin{proof}[Proof of Corollary \ref{main_theorem2}]
Evidently, we have (\ref{sup_sup_condition}) $\Rightarrow$ (\ref{sup_L1_function}). Moreover, 
(\ref{sup_L1_function}) $\Rightarrow$ (\ref{integral_eq_cond}) which completes the proof.
\end{proof}
\begin{proof}[Proof of Theorem \ref{thm:exact_error}]
Now we have 
\begin{equation}
\label{apu_esitys}
V_{n,Q_n}^D(a)=V_Q^D(a) +\frac{A_n(a)}
{n^\beta}+\frac{O_n(a)}{n^\kappa}
\end{equation}
for some function $O_n(a)$ which is bounded in $n$. 
Thus, by Theorem \ref{theor_v_general} and condition (\ref{exact_error_bounded_coefficient}) we only have to show that 
$$
\int_0^\infty f'(a)\frac{O_n(a)}{n^\kappa}\ud a = o(n^{-\beta}).
$$
Note that $O^*=:\sup_a\sup_n \lvert O_n(a)\rvert <\infty$. We have
\begin{equation*}
\begin{split}
&\frac{1}{n^{\kappa}}\int_0^\infty \lvert f'(a)O_n(a)\rvert \ud a\\
&=\frac{1}{n^{\kappa}}\int_0^{1} \lvert f'(a)O_n(a)\rvert \ud a\\
&+\frac{1}{n^{\kappa}}\int_{1}^\infty \lvert f'(a)O_n(a)\rvert \ud a\\
&\leq \frac{O^*\sup_{0\leq a\leq 1}\lvert f'(a)\rvert }{n^{\kappa}}\\
&+\frac{1}{n^{\kappa}}\int_{1}^\infty \lvert f'(a)O_n(a)\rvert \ud a.
\end{split}
\end{equation*}
Evidently, the first term is $o(n^{-\beta})$. For the second term, we obtain by H\"{o}lder inequality that 
\begin{equation*}
\begin{split}
&\frac{1}{n^{\kappa}}\int_{1}^\infty \lvert f'(a)O_n(a)\rvert \ud a\\
&\leq \left[\int_{1}^\infty \frac{1}{a^{\delta\frac{1+\theta}{\theta}}}\ud a\right]^{\frac{\theta}{1+\theta}}\left[\int_{1}^\infty\lvert a^\delta f'(a)\rvert ^{1+\theta}\frac{\lvert O_n(a)\rvert ^{1+\theta}}{n^{\kappa(1+\theta)}}\ud a\right]^{\frac{1}{1+\theta}}\\
&\leq c(\theta,\delta)\frac{\left(O^*\right)^{\frac{\theta}{1+\theta}}}{n^{\tilde{\kappa}}}\left[\sup_n\int_0^\infty\lvert a^\delta f'(a)\rvert ^{1+\theta}\frac{\lvert O_n(a)\rvert }{n^\kappa}\ud a\right]^{\frac{1}{1+\theta}}
\end{split}
\end{equation*}
where
$\tilde{\kappa} = \frac{\kappa\theta}{\theta + 1} > \beta$ by (\ref{cond_eps_delta}).
It remains to show that
$$
\sup_n\int_0^\infty\lvert a^\delta f'(a)\rvert ^{1+\theta}\frac{\lvert O_n(a)\rvert }{n^\kappa}\ud a < \infty.
$$
By (\ref{apu_esitys}), (\ref{finite_limit_price}) and (\ref{bounded_coef2}) this is equivalent to condition
$$
\sup_n \int_0^\infty\lvert a^\delta f'(a)\rvert ^{1+\theta}V_{n,Q_n}^D(a)\ud a < \infty.
$$
This follows from assumption that $X_T^n$ is $(B^n,F_{\delta,\theta})$-UI and hence we are done.
\end{proof}
Before proving the results related to binomial tree approximations, we start with a simplifying lemma.
\begin{lma}
\label{lemma_proof}
Let $f\in\Pi$ such that $f'$ is polynomially bounded. Then
\begin{equation}
\label{boundedness_n}
\int_0^\infty \lvert af'(a)\rvert^{p}\lvert e^{-\frac{d_2^2}{2}}d_1^3\rvert\ud a < \infty
\end{equation}
for every $p\geq 1$.
\end{lma}
\begin{proof}[Proof of Lemma \ref{lma:ok_pricing}]
Evidently we have 
$$
f\in\bigcap_n \Pi_{BT(n)}(S_T^n),
$$
where $BT(n)$ denotes binomial tree with $n$ number of steps. 
Hence we only have to prove that $f\in\Pi_{BS}(S_T)$, where $BS$ denotes Black-Scholes model. In order to have this it is sufficient 
to show that
\begin{equation}
\label{BS-ok}
\lvert\int_0^\infty f'(a)V_{BS}^D(a)\ud a \rvert < \infty,
\end{equation}
where $V_{BS}^D(a) = e^{-rT}\Phi(d_2)$. Note that
for every number $x>0$, we have
$$
1 - \Phi(x)\leq \frac{1}{\sqrt{2\pi}x}e^{-\frac{x^2}{2}}.
$$
Hence for sufficiently large $a$, we have
\begin{equation}
\label{upper_bound3}
\Phi(d_2)\leq \frac{1}{\sqrt{2\pi}\lvert d_2\rvert }e^{-\frac{d_2^2}{2}}.
\end{equation}
The statement follows from Lemma \ref{lemma_proof}.
\end{proof}
\begin{proof}[Proof of Theorem \ref{theor_rate_general}]
For sufficiently large $a$, we have
\begin{equation}
\label{upper_bound1}
e^{-\frac{d_2^2}{2}}\left\lvert B_n - \frac{d_2\Delta_n}{2}\right\rvert  \leq ce^{-\frac{d_2^2}{2}}\lvert d_1\rvert ^3,
\end{equation}
and
\begin{equation}
\label{upper_bound2}
e^{-\frac{d_2^2}{2}}\left\lvert \Delta_n\right\rvert  \leq ce^{-\frac{d_2^2}{2}}\lvert d_1\rvert ^3
\end{equation}
for some constant $c$. Now the result follows from Theorem \ref{thm:exact_error}, Lemma \ref{lma:ok_pricing}, 
Theorem \ref{theor_rate_call} and Remark \ref{rate_rema}. 
\end{proof}
\begin{proof}[Proof of Corollary \ref{theor_rate_c2}]
The proof is essentially the same as the proof of Theorem \ref{theor_rate_general} and the details are left to the reader.
\end{proof}
The proof for smooth convergence is based on the following well known auxiliary result (for instance, see \cite{atkinson}).
\begin{lma}
\label{lma:error_trapet}
Let $I_{t,n}$ denote the trapezoidal approximation with $n$ equidistant points for integral
$$
\int_0^b g(a)\ud a
$$
given by (\ref{trapezoid_appro}). If $g$ is two times continuously differentiable on $[0,b]$, then the error for trapezoidal 
approximation satisfy
$$
\left\lvert I_{t,n} - \int_0^b g(a)\ud a \right\rvert \leq \frac{b^3}{12n^2}\sup_{x\in[0,b]}\lvert g''(x)\rvert .
$$
\end{lma}
\begin{proof}[Theorem \ref{thm:smooth}]
For every fixed $a$ we have (see \cite{Palmer})
$$
\lvert\lambda_n(a)\rvert \leq \frac{1}{\sigma\sqrt{Tn}}.
$$
Hence, since $f$ is polynomially bounded, we have
\begin{equation}
\label{apu_integraali}
\int_0^\infty f'(a)B_{n,\lambda(a)}(a)\ud a = C + O\left(\frac{1}{\sqrt{n}}\right),
\end{equation}
where $B_{n,\lambda(a)}(a)$ is given in Appendix A with $\lambda_n(a)$ given by (\ref{lambda_centered}) and $C$ is given by 
(\ref{smooth_constant}).
From Theorem \ref{theor_rate_call} we obtain that
\begin{equation*}
\begin{split}
V_{t,n}^f &= \frac{n^{\alpha-1}}{2}\sum_{k=0}^{n-1}\left(f'(kn^{\alpha-1})V_{BS}^D(kn^{\alpha-1})+ f'((k+1)n^{\alpha-1})
V_{BS}^D(kn^{\alpha-1})\right)\\
&+\frac{1}{n}\cdot\frac{n^{\alpha-1}}{2}\sum_{k=0}^{n-1}\left(f'(kn^{\alpha-1})B_{n,\lambda(a)}(kn^{\alpha-1})+ f'((k+1)n^{\alpha-1})
B_{n,\lambda(a)}(kn^{\alpha-1})\right)\\
&+ \frac{1}{n\sqrt{n}}\cdot\frac{n^{\alpha-1}}{2}\sum_{k=0}^{n-1}\left(f'(kn^{\alpha-1})O_n(kn^{\alpha-1})+ f'((k+1)n^{\alpha-1})
O_n(kn^{\alpha-1})\right)\\
&=: I_1 + \frac{I_2}{n} + \frac{I_3}{n\sqrt{n}},
\end{split}
\end{equation*}
where $I_1$ and $I_2$ are trapezoidal approximations for integrals
$$
\int_0^{n^{\alpha}} f'(a)V_{BS}^D(a)\ud a
$$
and
$$
\int_0^{n^{\alpha}} f'(a)B_{n,\lambda(a)}(a)\ud a.
$$
Since $\alpha<\frac{1}{3}$, we obtain by (\ref{apu_integraali}) and Lemma \ref{lma:error_trapet} that 
$$
I_1 - \int_0^{n^{\alpha}} f'(a)V_{BS}^D(a)\ud a = o(n^{-1})
$$
and
$$
\frac{I_2}{n} - \frac{1}{n}\int_0^{n^{\alpha}} f'(a)C(a)\ud a = o(n^{-1}),
$$
where $C(a)$ is given by (\ref{smooth_const_function}). 
Hence we have
$$
V_{t,n}^f = V_{BS}^f + \frac{C}{n} - \int_{n^\alpha}^\infty f'(a)V_{BS}^D(a)\ud a 
- \frac{1}{n}\int_{n^\alpha}^\infty f'(a)C(a)\ud a + \frac{I_3}{n\sqrt{n}} + o(n^{-1}).
$$
It is straightforward to see that 
$$
\int_{n^\alpha}^\infty f'(a)V_{BS}^D(a)\ud a = o(n^{-1})
$$
and
$$
\int_{n^\alpha}^\infty f'(a)C(a)\ud a = o(n^{-1}).
$$
Hence we only have to show that $\frac{I_3}{n\sqrt{n}} = o(n^{-1})$ in order to complete the proof. 
This follows by arguments in the proof of Theorem \ref{thm:exact_error}, since $f'$ is polynomially bounded and we have
$$
\int_0^\infty a^pV_{n,\lambda(a)}^D(a)\ud a \rightarrow \int_0^\infty a^p V_{BS}^D(a)\ud a
$$
for every $p>0$. 
\end{proof}

\appendix
\section{List of constants for binomial error expansion}
\begin{equation*}
\begin{split}
u &= e^{\sigma\sqrt{\Delta t}+\lambda_n\Delta t},\quad d = e^{-\sigma\sqrt{\Delta t}+\lambda_n\Delta t}\\
d_1 &= \frac{\log\frac{S_0}{a} + \left(r+\frac{\sigma^2}{2}\right)T}{\sigma\sqrt{T}}, \quad d_2 = d_1 - \sigma\sqrt{T},\\
\Delta_n &= 1 - 2 frac\left[\frac{\log\frac{S_0}{a}+n\log d}{\log\frac{u}{d}}\right],\\
B_n &= \frac{d_1^3+d_1d_2^2+2d_2-4d_1}{24} +
\frac{(2-d_1d_2-d_1^2)\sqrt{T}}{6\sigma}(r-\lambda_n\sigma^2)\\
&\quad+\frac{Td_1}{2\sigma^2}(r-\lambda_n\sigma^2)^2,\\
\tilde{B}_n &= -\sigma^2T(6+d_1^2 +d_2^2) +
4T(d_1^2-d_2^2)(r-\lambda_n\sigma^2)-12T^2(r-\lambda_n\sigma^2)^2,\\
J_n &= \frac{e^{-rT}e^{-\frac{d_1^2}{2}}}{\sqrt{2\pi}}\left[\frac{\Delta_n}{\sqrt{n}} -
\frac{d_2\Delta_n^2}{2n}+\frac{B_n}{n}\right],\\
\widehat{J}_n &=
\frac{e^{-rT}e^{-\frac{d_1^2}{2}}}{\sqrt{2\pi}}\left[\frac{\Delta_n-2}{\sqrt{n}} -
\frac{d_2(\Delta_n-2)^2}{2n}+\frac{B_n}{n}\right],\\
H_n &= \frac{S_0e^{-\frac{d_1^2}{2}}}{24\sigma\sqrt{2\pi T}} \left[\tilde{B}_n - 12\sigma^2T(\Delta_n^2 - 1)\right].
\end{split}
\end{equation*}

\bibliographystyle{plain}
\bibliography{bibli_rate}

\begin{thebibliography}{10}

\bibitem{atkinson}
K.~E. Atkinson.
\newblock {\em An Introduction to Numerical Analysis}.
\newblock John Wiley \& Sons, 1989.

\bibitem{Vostrikova}
S.~Cawston and L.~Vostrikova.
\newblock Continuity property of option prices in incomplete markets.
\newblock {\em Theory of Probability and their applications}, accepted:arXiv:
  0903.3274, 2012.

\bibitem{Palmer}
L.~Chang and K.~Palmer.
\newblock Smooth convergence in the binomial model.
\newblock {\em Finance and Stochastics}, 11:91--105, 2007.

\bibitem{Cox}
J.~Cox, S.A. Ross, and M.~Rubinstein.
\newblock Option pricing: a simplified approach.
\newblock {\em J. Financ. Econ.}, 7:229--263, 1979.

\bibitem{Diener}
F.~Diener and M.~Diener.
\newblock Asymptotics of the price oscillations of a european call option in a
  tree model.
\newblock {\em Mathematical Finance}, 14:271--293, 2004.

\bibitem{Heston}
S.~Heston and G.~Zhou.
\newblock On the rate of convergence of discrete-time contingent claims.
\newblock {\em Mathematical Finance}, 10:53--75, 2000.

\bibitem{Hsia}
C.-C. Hsia.
\newblock On binomial option pricing.
\newblock {\em J. Financ. Res.}, 6:41--50, 1983.

\bibitem{jacod}
J.~Jacod and P.~Protter.
\newblock {\em Probability Essentials}.
\newblock Springer, 2004.

\bibitem{Jarrow}
M.J. Jarrow.
\newblock A characterization theorem for unique risk neutral probability
  measures.
\newblock {\em Economic Letters}, 22:61--65, 1986.

\bibitem{Joshi}
M.~Joshi.
\newblock Achieving higher order convergence for the prices of european options
  in binomial trees.
\newblock {\em Mathematical Finance}, 20:89--103, 2010.

\bibitem{Leisen}
D.~Leisen and M.~Reimer.
\newblock Binomial models for option valuation -- examining and improving
  convergence.
\newblock {\em Applied Mathematical Finance}, 3:319--346, 1996.

\bibitem{lasse}
L.~Leskel\"{a} and M.~Vihola.
\newblock Stochastic order characterization of uniform integrability and
  tightness.
\newblock {\em Statistics \& Probability Letters}, doi:
  10.1016/j.spl.2012.09.023, 2012.

\bibitem{Rendleman}
R.~Rendleman and B.~Bartter.
\newblock Two state option pricing.
\newblock {\em J. Financ.}, 34:1093--1110, 1979.

\bibitem{Tian}
Y.~Tian.
\newblock A modified lattice approach to option pricing.
\newblock {\em Journal of Futures Markets}, 13:563--577, 1993.

\bibitem{viitasaari}
L.~Viitasaari.
\newblock Option prices with call prices.
\newblock {\em submitted}, arXiv: 1207.6205, 2012.

\bibitem{Walsh}
J.~Walsh.
\newblock The rate of convergence of the binomial tree scheme.
\newblock {\em Finance and Stochastics}, 7:337--361, 2003.

\bibitem{Xiao}
X.~Xiao.
\newblock Improving speed of convergence for the prices of european options in
  binomial trees with even numbers of steps.
\newblock {\em App. Math. and Computation}, 216:2659--2670, 2010.

\end{thebibliography}

\end{document}